\documentclass{amsart}
\usepackage[french]{babel}
\usepackage[utf8]{inputenc}
\usepackage[a4paper,centering,textwidth=145mm,textheight=195mm]{geometry}

\newcommand{\bbZ}{\mathbb{Z}}
\newcommand{\Symd}{\operatorname{Symd}}
\newcommand{\Sym}{\operatorname{Sym}}
\newcommand{\Prp}{\operatorname{Prp}}
\newcommand{\Pcrd}{\operatorname{Pcrd}}
\newcommand{\Trp}{\operatorname{Trp}}
\newcommand{\Srp}{\operatorname{Srp}}
\newcommand{\Nrp}{\operatorname{Nrp}}
\newcommand{\Trd}{\operatorname{Trd}}
\newcommand{\Srd}{\operatorname{Srd}}
\newcommand{\Nrd}{\operatorname{Nrd}}

\newcommand{\End}{\operatorname{End}}
\newcommand{\invo}{\overline{\rule{2.5mm}{0mm}\rule{0mm}{4pt}}}
\newcommand{\sq}{q}
\newcommand{\diag}{\operatorname{diag}}

\newtheorem{theo}{Théorème}[section]
\newtheorem{lemm}[theo]{Lemme}
\newtheorem{coro}[theo]{Corollaire}
\newtheorem{prop}[theo]{Proposition}

\theoremstyle{definition}
\newtheorem{exam}[theo]{Exemple}

\title{Invariants de Witt des involutions de bas degré en
  caractéristique~2}

\author{Jean-Pierre Tignol}
\address{ICTEAM, UCLouvain, 4 avenue G. Lemaître, boîte L4.05.01,
  B-1348 Louvain-la-Neuve, Belgique}

\email{Jean-Pierre.Tignol@uclouvain.be}

\thanks{Ce travail a bénéficié d'une subvention du Fonds de la
  Recherche Scientifique--FNRS portant la référence~CDR~J.0159.19.}


\begin{document}
\begin{abstract} 
  À toute involution symplectique sur une algèbre simple centrale de
  degré~$8$ sur un corps de caractéristique~$2$ sont associées de
  manière canonique une $3$-forme de Pfister et une $5$-forme de
  Pfister quadratiques, qui détiennent des informations sur la
  structure de l'algèbre à involution. La même construction associe
  une $2$-forme de Pfister quadratique et une $4$-forme de Pfister
  quadratique à toute involution unitaire et une quasi $1$-forme de
  Pfister et une quasi $3$-forme de Pfister à toute involution
  orthogonale sur une algèbre simple centrale de degré~$4$.  
\end{abstract}

\maketitle

Dans tout ce texte, $A$ désigne une algèbre simple centrale sur un
corps $F$ et $\sigma$ une involution symplectique sur $A$,
c'est-à-dire un anti-automorphisme d'ordre~$2$ qui après extension des
scalaires à un corps de déploiement est adjoint à une forme bilinéaire
alternée non dégénérée. On note
\[
  \Symd(\sigma) = \{x+\sigma(x)\mid a\in A\}
\]
et, pour $a\in\Symd(\sigma)$, on note $\Prp_{\sigma, a}(X)$ le
polynôme pfaffien réduit de $a$, qui est le polynôme unitaire dont le
carré est le polynôme caractéristique réduit $\Pcrd_a(X)$. Si $\deg
A=2m$, alors
\[
  \Prp_{\sigma,a}(X) = X^m
  -\Trp_\sigma(a)X^{m-1}+\Srp_\sigma(a)X^{m-2}-\cdots
  +(-1)^m\Nrp_\sigma(a)
\]
pour $\Trp_\sigma$, $\Srp_\sigma$ et $\Nrp_\sigma\colon\Symd(\sigma)\to
F$ des formes de degré~$1$, $2$ et $m$ respectivement.

\begin{theo}
  \label{th:symp}
  Si $\deg A=8$ et la caractéristique de $F$ est~$2$, alors
  il existe une $3$-forme quadratique de Pfister $\pi_3$ et une
  $5$-forme quadratique de Pfister $\pi_5$ déterminées de manière
  unique à isométrie près par la propriété suivante: la classe de
  $\Srp_\sigma$ dans le groupe de Witt $I_qF$ se
  décompose comme suit:
  \[
    \Srp_\sigma= [1,1] + \pi_3+\pi_5
  \]
  où $[1,1]$ est la forme $X^2+XY+Y^2$. De plus,
  \begin{enumerate}
  \item[\emph{(i)}]
    $\pi_5$ est multiple de $\pi_3$, c'est-à-dire qu'il existe $a_1$,
  $a_2\in F^\times$ tels que
  $\pi_5=\langle1,a_1,a_2,a_1a_2\rangle\pi_3$;
\item[\emph{(ii)}]
  la forme $\pi_3$ est hyperbolique si et seulement si l'algèbre
  $A$ se décompose en produit tensoriel d'algèbres de
  quaternions stables sous $\sigma$;
\item[\emph{(iii)}]
  la forme $\pi_5$ est hyperbolique si et seulement si $\Symd(\sigma)$
  contient un élément non central dont le carré est central.
  \end{enumerate}
\end{theo}

La démonstration est donnée dans la section~\ref{sec:dem}: voir le
corollaire~\ref{coro:comp} et les propositions~\ref{prop:pi3} et
\ref{prop:pi5}. On peut cependant démontrer d'emblée 
l'unicité des formes $\pi_3$ et $\pi_5$: si
$\Srp_\sigma=[1,1]+\pi_3+\pi_5 = [1,1]+\pi'_3+\pi'_5$, alors
$\pi_3-\pi'_3=\pi'_5-\pi_5\in I^5_qF$; mais la classe de Witt de
$\pi_3-\pi'_3$ est représentée par une forme de dimension~$16$, donc
le \emph{Hauptsatz} d'Arason--Pfister \cite[Th.~23.7]{EKM} entraîne
$\pi_3=\pi'_3$, donc aussi $\pi_5=\pi'_5$.

La démonstration des autres assertions repose de manière essentielle
sur l'existence d'algèbres biquadratiques étales dans $\Symd(\sigma)$,
qui est démontrée dans \cite{BGBT1} et exploitée de manière analogue
dans~\cite{BGBT3}. Dans cette dernière référence, on trouve une
construction de la forme $\pi_3$ 
(qui y est appelée \emph{forme de Pfister discriminante}),
ainsi qu'une preuve du critère de décomposabilité correspondant, en
toute caractéristique. Dans le présent travail, la restriction à la
caractéristique~$2$ permet d'importantes simplifications ainsi que la
définition de la forme~$\pi_5$, qui ne semble pas admettre d'analogue
en caractéristique différente de~$2$.

Les techniques utilisées dans la démonstration du
théorème~\ref{th:symp} peuvent 
aussi servir pour les involutions unitaires et orthogonales sur les
algèbres simples centrales de degré~$4$: les résultats correspondants
sont détaillés dans les théorèmes~\ref{th:unit} et \ref{th:orth}
énoncés dans la dernière section. Comme dans le cas
symplectique, \cite{BGBT3} présente des constructions analogues, mais
limitées aux formes de Pfister
\guillemotleft\,discriminantes\,\guillemotright, en toute 
caractéristique 
dans le cas unitaire et en caractéristique différente de~$2$ dans le
cas orthogonal.

\section{La forme $\Srp_\sigma$}

Dans cette section, $\deg A=2m$ avec $m\geq1$ et la caractéristique de
$F$ est arbitraire. On note $\Trd_A$ la trace réduite de $A$ et
$\Srd_A$ la forme quadratique sur $A$ qui donne le
coefficient du terme de degré $2m-2$ du polynôme caractéristique réduit.

\begin{lemm}
  \label{lem:Srp}
  Pour $x$, $y\in \Symd(\sigma)$ avec $x=x'+\sigma(x')$ et
  $y=y'+\sigma(y')$, on a $\Trp_\sigma(x)=\Trd_A(x')$ et 
  \[
    \Srp_\sigma(x+y) - \Srp_\sigma(x) -
    \Srp_\sigma(y) = \Trp_\sigma(x)\Trp_\sigma(y)
    - \Trd_A(xy').
  \]
\end{lemm}

\begin{proof}
  Il suffit d'établir ces relations après extension des scalaires à un
  corps de déploiement. On peut donc supposer que $A$ est une algèbre
  de matrices et que $\sigma$ est adjointe à une forme alternée non
  dégénérée. Alors $x'$ et $y'$ sont obtenus par spécialisation de
  matrices génériques; il suffit donc de montrer que les
  relations valent lorsque $x'$ et $y'$ sont des matrices génériques
  sur $\bbZ$. En comparant les
  coefficients des deux membres de l'équation
  \[
    (X^m-\Trp_\sigma(x)X^{m-1}+\Srp_\sigma(x)X^{m-2}-\cdots)^2 =
    X^{2m}-\Trd_A(x)X^{2m-1}+\Srd_A(x)X^{2m-2}-\cdots
  \]
  on obtient
  \begin{equation}
    \label{eq:1}
    2\Trp_\sigma(x)=\Trd_A(x) \qquad\text{et}\qquad
    2\Srp_\sigma(x)+\Trp_\sigma(x)^2=\Srd_A(x).
  \end{equation}
  Or, $\Trd_A(x)=\Trd_A(x'+\sigma(x'))=2\Trd_A(x')$, donc
  $2\Trp_\sigma(x)=2\Trd_A(x')$. Comme $2$ n'est pas un diviseur de
  zéro dans $\bbZ$, il en résulte $\Trp_\sigma(x)=\Trd_A(x')$.

  Par ailleurs, la relation suivante est établie en~\cite[(0.2)]{BoI}:
  \[
    \Srd_A(x+y)-\Srd_A(x)-\Srd_A(y) = \Trd_A(x)\Trd_A(y)-
    \Trd_A(xy).
  \]
  En y remplaçant $\Srd_A(x+y)$, $\Srd_A(x)$ et $\Srd_A(y)$
  par les expressions données par la deuxième équation
  de~\eqref{eq:1}, on obtient
  \[
    2\Srp_\sigma(x+y)-2\Srp_\sigma(x)-2\Srp_\sigma(y) +
    2\Trp_\sigma(x)\Trp_\sigma(y) =
    \Trd_A(x)\Trd_A(y)-\Trd_A(xy).
  \]
  Vu la première équation de~\eqref{eq:1}, et vu que $\Trd_A(xy) =
  \Trd_A(xy') + \Trd_A(x\sigma(y'))=2\Trd_A(xy')$ (la dernière égalité
  résultant du fait que $x=\sigma(x)$), on en déduit
  \[
    2\Srp_\sigma(x+y)-2\Srp_\sigma(x)-2\Srp_\sigma(y) =
    2\Trp_\sigma(x) \Trp_\sigma(y) - 2\Trd_A(xy').
  \]
  La deuxième formule de l'énoncé en découle car $2$ n'est pas
  diviseur de zéro dans $\bbZ$.
\end{proof}

Soit $b_\sigma\colon\Symd(\sigma)\times\Symd(\sigma)\to F$ la forme
polaire de $\Srp_\sigma$, donnée par
\[
  b_\sigma(x,y) = \Srp_\sigma(x+y)-\Srp_\sigma(x) -
  \Srp_\sigma(y)
  \quad\text{pour $x$, $y\in\Symd(\sigma)$.}
\]
La forme quadratique $\Srp_\sigma$ est dite non singulière lorsque le
radical de sa forme polaire $b_\sigma$ est réduit à $\{0\}$.

\begin{coro}
  \label{coro:nondeg}
  Si la caractéristique de $F$ ne divise pas $m-1$, la forme
  $\Srp_\sigma$ est non singulière.
\end{coro}

\begin{proof}
  Si $x$ est dans le radical de $b_\sigma$, c'est-à-dire que
  $b_\sigma(x,y)=0$ pour tout $y\in\Symd(\sigma)$, alors le
  lemme~\ref{lem:Srp} donne
  \[
    \Trp_\sigma(x)\Trd_A(y') - \Trd_A(xy')=0
    \qquad\text{pour tout $y'\in A$.}
  \]
  Alors $\Trd_A\bigl((\Trp_\sigma(x)-x)y'\bigr)=0$ pour tout $y'\in
  A$. Comme le radical de la forme bilinéaire $\Trd_A(XY)$ est nul, il
  en découle $x=\Trp_\sigma(x)\in F$. Or, $\Trp_\sigma(x)=mx$
  pour $x\in F$, donc la dernière équation entraîne $(m-1)x=0$,
  d'où $x=0$ si $m-1$ est inversible dans $F$.
\end{proof}

\section{Décomposition orthogonale}
\label{sec:orth}

Dans cette section, on suppose $\deg A=8$ mais on n'impose aucune
restriction sur la caractéristique de $F$. D'après~\cite[Th.~7.4,
Th.~4.1]{BGBT1} on peut trouver dans $\Symd(\sigma)$ une $F$-algèbre
étale $L$ qui est produit tensoriel de deux $F$-algèbres étales
quadratiques et telle que $A$ est libre comme $L$-module. Comme le
Vierergruppe est le seul sous-groupe abélien élémentaire du groupe de
permutations de quatre éléments qui agisse
transitivement, il y a une unique manière (à automorphisme du groupe
près) de
définir sur $L$ une action d'un groupe $G$ abélien élémentaire qui en
fait une $F$-algèbre $G$-galoisienne. Soit
\[
  G=\{1,\,\alpha_1,\,\alpha_2,\,\alpha_3\} \qquad\text{où
    $\alpha_1^2=\alpha_2^2=\alpha_3^2=1$}.
\]
Pour $i=1$, $2$, $3$ on note $L_i=L^{\alpha_i}$ la sous-algèbre de $L$
fixe sous $\alpha_i$. C'est une $F$-algèbre quadratique étale, dont on
note $T_i\colon L_i\to F$ et $N_i\colon L_i\to F$ la trace et la
norme. On définit aussi
\[
  W_i=\{x\in\Symd(\sigma)\mid x\ell = \alpha_i(\ell)x \text{ pour tout
    $\ell\in L$}\}.
\]
La multiplication dans $A$ fait de chaque $W_i$ un module à droite sur
$L_i$. 

\begin{prop}
  \label{prop:decorth}
  Pour la forme polaire $b_\sigma$ de $\Srp_\sigma$, l'espace
  $\Symd(\sigma)$ se décompose en somme orthogonale:
  $\Symd(\sigma)=L\stackrel{\perp}{\oplus} W_1
  \stackrel{\perp}{\oplus} W_2 \stackrel{\perp}{\oplus}W_3$.
  De plus, pour chaque $i=1$, $2$, $3$ le $L_i$-module $W_i$ est libre
  de rang~$4$ et l'élévation au carré $x\mapsto x^2$ définit une forme
  quadratique non singulière $\sq_i\colon W_i\to L_i$. Pour $x\in W_i$
  on a $\Trp_\sigma(x)=0$ et 
  $\Srp_\sigma(x)=-T_i(x^2)$, quel que soit $i=1$, $2$, $3$.
\end{prop}

\begin{proof}
  Il suffit de voir que les énoncés valent après extension des
  scalaires. On peut donc supposer que $A$ et $L$ sont déployées: soit
  $A=\End_FV$ pour un espace vectoriel $V$ de dimension~$8$ et
  $L=Fp_1\oplus Fp_2\oplus Fp_3\oplus Fp_4$ où $p_1$, \ldots, $p_4$
  sont des projections de $V$ sur des sous-espaces $V_1$, \ldots,
  $V_4$ supplémentaires, c'est-à-dire que $V=V_1\oplus\cdots\oplus
  V_4$. Comme $A$ est libre comme $L$-module, les $V_i$ sont tous de
  même dimension, donc $\dim V_i=2$ pour tout~$i$.

  L'involution $\sigma$ est adjointe à une forme bilinéaire alternée
  non singulière sur $V$. Comme chaque $p_i$ est symétrique sous
  $\sigma$, les espaces $V_i$ sont orthogonaux deux à deux. En
  concaténant des bases symplectiques de ceux-ci, on obtient une base
  symplectique de $V$, par rapport à laquelle on peut représenter $A$
  par des matrices: $A=M_8(F)$. Pour la facilité des calculs, on
  décompose les matrices en blocs d'ordre~$2$, ce qui conduit à
  représenter $A=M_4(Q)$ où $Q=M_2(F)$ est l'algèbre de quaternions
  déployée. Dans cette représentation, $\sigma$ est donnée par
  $\sigma(a)=\overline{a}^t$, où $t$ est la transposition et $\invo$
  est la conjugaison de $Q$ (qui est son unique involution
  symplectique), agissant sur chaque entrée des matrices de
  $M_4(Q)$. L'algèbre $L$ est alors identifiée à l'algèbre diagonale
  dont les entrées sont dans $F$: pour $x_1$, \ldots, $x_4\in F$
  \[
    x_1p_1+x_2p_2+x_3p_3+x_4p_4 = \diag(x_1,\,x_2,\,x_3,\,x_4)
    \in M_4(Q).
  \]
  En utilisant $\Symd(\invo)=F$ dans $Q$, on vérifie sans peine:
  \[
    \Symd(\sigma) = \left\{
      \begin{pmatrix}
        x_1&x_{12}&x_{13}&x_{14}\\
        \overline{x_{12}}&x_2&x_{23}&x_{24}\\
        \overline{x_{13}}&\overline{x_{23}}&x_3&x_{34}\\
        \overline{x_{14}}&\overline{x_{24}}&\overline{x_{34}}&x_4
      \end{pmatrix}
      \mid
      x_i\in F \text{ et } x_{ij}\in Q
      \right\}
      .
    \]
    Quitte à renuméroter les $p_i$, on peut supposer que $\alpha_1$
    échange $p_1$ et $p_2$ (donc aussi $p_3$ et $p_4$) et que
    $\alpha_2$ échange $p_1$ et $p_3$ (donc aussi $p_2$ et
    $p_4$). Alors
    \begin{gather*}
      L_1=\{\diag(x_1,\,x_1,\,x_2,\,x_2)\mid x_1,\;x_2\in F\},
      \qquad
      L_2=\{\diag(x_1,\,x_2,\,x_1,\,x_2)\mid x_1,\;x_2\in F\},
      \\
      L_3=\{\diag(x_1,\,x_2,\,x_2,\,x_1)\mid x_1,\;x_2\in F\}.
    \end{gather*}
    Le calcul donne
    \[
      W_1 = \left\{
      \begin{pmatrix}
        0&x_{12}&&\\
        \overline{x_{12}}&0&&\\
        &&0&x_{34}\\
        &&\overline{x_{34}}&0
      \end{pmatrix}
      \mid
      x_{12}, x_{34}\in Q
    \right\},
  \]
  \[
    W_2 = \left\{
      \begin{pmatrix}
        &&x_{13}&0\\
        &&0&x_{24}\\
        \overline{x_{13}}&0&&\\
        0&\overline{x_{24}}&&
      \end{pmatrix}
      \mid
      x_{13}, x_{24}\in Q
    \right\}
  \]
  et
  \[
    W_3 = \left\{
      \begin{pmatrix}
        &&0&x_{14}\\
        &&x_{23}&0\\
        0&\overline{x_{23}}&&\\
        \overline{x_{14}}&0&&
      \end{pmatrix}
      \mid
      x_{14}, x_{23}\in Q
    \right\}.
  \]
  Il est donc clair que $\Symd(\sigma)=L\oplus W_1\oplus W_2\oplus
  W_3$. On voit aussi que pour tout $i=1$, $2$, $3$ le $L_i$-module
  $W_i$ est isomorphe à $Q_{L_i}\simeq M_2(F)\times M_2(F)$ et que
  $\sq_i$ est une forme quadratique isométrique à $(n_Q)_{L_i}$, où
  $n_Q\colon Q\to F$ est la norme réduite (qui est le déterminant de
  $M_2(F)$); cette forme quadratique est donc non singulière.

  Pour $x\in W_i$ on peut
  écrire $x=x'+\sigma(x')$ avec $x'\in M_4(Q)$ une matrice
  triangulaire dont la diagonale est nulle. Le lemme~\ref{lem:Srp}
  donne alors $\Trp_\sigma(x)=\Trd_A(x')=0$. Pour $y\in W_j$ avec
  $j\neq i$ ou $y\in L$ on a aussi $\Trd_A(x'y)=0$, donc le
  lemme~\ref{lem:Srp} donne aussi $b_\sigma(x,y)=0$, ce qui établit
  que $L$, $W_1$, $W_2$ et $W_3$ sont orthogonaux deux à deux.

  Un calcul direct montre que le polynôme caractéristique réduit de $
  \bigl(\begin{smallmatrix}
    0&x\\ \overline x&0
  \end{smallmatrix}\bigr)\in M_2(Q)$ est $(X^2-n_Q(x))^2$. Dès
  lors, pour
  \[
    x=
    \begin{pmatrix}
        0&x_{12}&&\\
        \overline{x_{12}}&0&&\\
        &&0&x_{34}\\
        &&\overline{x_{34}}&0
      \end{pmatrix}\in W_1,\qquad\text{avec $x_{12}$, $x_{34}\in Q$,}
    \]
    on a $\Prp_{\sigma,x}(X)=(X^2-n_Q(x_{12}))(X^2-n_Q(x_{34}))$, donc
    $\Srp_\sigma(x)=-n_Q(x_{12})-n_Q(x_{34}) = -T_1(x^2)$. On démontre
    de même que $\Srp_\sigma(x)=-T_i(x^2)$ pour $x\in W_i$ avec $i=2$
    ou $3$.
\end{proof}

\section{Démonstration du théorème~\ref{th:symp}}
\label{sec:dem}

Poursuivant avec les mêmes notations que dans la section précédente,
on suppose à présent que la caractéristique de~$F$ est~$2$. D'après le
corollaire~\ref{coro:nondeg}, la forme quadratique $\Srp_\sigma$ est
non singulière; il en est donc de même de ses restrictions aux espaces
$W_i$, vu la proposition~\ref{prop:decorth}.

\begin{prop}
  \label{prop:compo}
  Pour $x_1\in W_1$ et $x_2\in W_2$, on pose
  $x_1*x_2=x_1x_2+x_2x_1$. Alors $x_1*x_2\in W_3$ et
  $\Srp_\sigma(x_1*x_2) = \Srp_\sigma(x_1)\Srp_\sigma(x_2)$.
\end{prop}

\begin{proof}
  Il suffit de voir que l'énoncé vaut après extension des
  scalaires. On peut donc supposer que $L$ et $A$ sont déployées, et
  reprendre les notations de la démonstration de la
  proposition~\ref{prop:decorth}. Si $x_{12}$, $x_{34}$, $x_{13}$,
  $x_{24}\in Q$ sont tels que
  \[
    x_1=
    \begin{pmatrix}
        0&x_{12}&&\\
        \overline{x_{12}}&0&&\\
        &&0&x_{34}\\
        &&\overline{x_{34}}&0
      \end{pmatrix}\in W_1
      \quad\text{et}\quad
      x_2=
      \begin{pmatrix}
        &&x_{13}&0\\
        &&0&x_{24}\\
        \overline{x_{13}}&0&&\\
        0&\overline{x_{24}}&&
      \end{pmatrix}\in W_2,
    \]
    alors
    \[
      x_1*x_2 =
      \begin{pmatrix}
        &&&x_{12}x_{24}+x_{13}x_{34}\\
        &&\overline{x_{12}}x_{13}+x_{24}\overline{x_{34}}&\\
        &x_{34}\overline{x_{24}}+\overline{x_{13}}x_{12}&&\\
        \overline{x_{34}}\;\overline{x_{13}} + \overline{x_{24}}\;
        \overline{x_{12}} &&&
      \end{pmatrix}\in W_3.
    \]
    De plus, il ressort de la preuve de la
    proposition~\ref{prop:decorth} que
    \[
      \Srp_\sigma(x_1)=n_Q(x_{12})+n_Q(x_{34}),\qquad
      \Srp_\sigma(x_2)=n_Q(x_{13})+n_Q(x_{24})
    \]
    et
    \[
      \Srp_\sigma(x_1*x_2) = n_Q(x_{12}x_{24}+x_{13}x_{34}) +
      n_Q(\overline{x_{12}}x_{13}+x_{24}\overline{x_{34}}).
    \]
    On développe le second membre de cette dernière équation: en
    notant $t_Q\colon Q\to F$ la trace réduite (qui est la trace de
    $M_2(F)$), 
    \begin{multline*}
      \Srp_\sigma(x_1*x_2) =
      \\
      n_Q(x_{12}x_{24}) + n_Q(x_{13}x_{34}) +
      t_Q(x_{12}x_{24}\overline{x_{34}}\;\overline{x_{13}}) +
      n_Q(\overline{x_{12}}x_{13}) + n_Q(x_{24}\overline{x_{34}}) +
      t_Q(\overline{x_{12}}x_{13}x_{34}\overline{x_{24}}).
    \end{multline*}
    Or, $t_Q(x_{12}x_{24}\overline{x_{34}}\;\overline{x_{13}})=
    t_Q(\overline{x_{12}}x_{13}x_{34}\overline{x_{24}})$. Comme la
    caractéristique est~$2$, l'expression de $\Srp_\sigma(x_1*x_2)$ se
    simplifie:
    \[
      n_Q(x_{12}x_{24}) + n_Q(x_{13}x_{34})+
      n_Q(\overline{x_{12}}x_{13}) + n_Q(x_{24}\overline{x_{34}}) =
      \bigl(n_Q(x_{12})+n_Q(x_{34})\bigr) \cdot
      \bigl(n_Q(x_{13})+n_Q(x_{24})\bigr).
      \qedhere
    \]  
\end{proof}

\begin{coro}
  \label{coro:comp}
  Il existe une $3$-forme quadratique de Pfister~$\pi_3$ et des
  éléments $a_1$, $a_2\in F^\times$ tels que les restrictions de
  $\Srp_\sigma$ à $W_1$, $W_2$ et $W_3$ soient liées par
  \[
    \Srp_\sigma\rvert_{W_1}\simeq \langle a_1\rangle \pi_3,
    \qquad
    \Srp_\sigma\rvert_{W_2}\simeq \langle a_2\rangle \pi_3,
    \qquad
    \Srp_\sigma\rvert_{W_3}\simeq \langle a_1a_2\rangle \pi_3.
  \]
  Pour la forme $\pi_3$ et $a_1$, $a_2$ ci-dessus, la classe de Witt
  de $\Srp_\sigma$ se décompose comme suit: $\Srp_\sigma=[1,1]+\pi_3 +
  \langle1,a_1,a_2,a_1a_2\rangle \pi_3$.
\end{coro}

\begin{proof}
  La proposition~\ref{prop:compo} montre que $*$ est une composition
  de formes quadratiques non singulières de dimension~$8$:
  $(W_1,\Srp_\sigma\rvert_{W_1}) \times
  (W_2,\Srp_\sigma\rvert_{W_2})\to (W_3,\Srp_\sigma\rvert_{W_3})$ au
  sens de~\cite[\S2]{BT}. Les formes
  $\Srp_\sigma\rvert_{W_i}$ sont donc multiples scalaires d'une même
  $3$-forme 
  de Pfister $\pi_3$, d'après~\cite[Prop.~2.9]{BT}. Si $x_1\in W_1$ et
  $x_2\in W_2$ sont des 
  vecteurs non nuls et non
  isotropes, alors en posant $a_i=\Srp_\sigma(x_i)$ pour $i=1$, $2$,
  on a $\Srp_\sigma\rvert_{W_i}\simeq\langle a_i\rangle\pi_3$ pour
  $i=1$, $2$. De plus, $\Srp_\sigma\rvert_{W_3}$ représente $a_1a_2$
  puisque $\Srp_\sigma(x_1*x_2)=\Srp_\sigma(x_1)\Srp_\sigma(x_2)$,
  donc $\Srp_\sigma\rvert_{W_3}\simeq\langle a_1a_2\rangle \pi_3$.
  Vu la proposition~\ref{prop:decorth}, on obtient
  \[
    \Srp_\sigma\simeq \Srp_\sigma\rvert_L \perp \langle a_1\rangle
    \pi_3 \perp \langle a_2\rangle \pi_3 \perp \langle a_1a_2\rangle
    \pi_3,
  \]
  et il ne reste plus qu'à voir que $\Srp_\sigma\rvert_L$ est
  équivalente au sens de Witt à $[1,1]$. Cela résulte du calcul
  général de la \guillemotleft\, seconde trace\,\guillemotright\ des
  algèbres étales dû à Bergé et Martinet~\cite[Th.~3.5]{BM}, mais en
  l'occurrence on peut en donner une preuve explicite comme suit.

  Si $\ell_1\in L_1$ est tel que $T_1(\ell_1)=1$, alors
  \[
    \Prp_{\sigma,\ell_1}(X) = \bigl(X^2+X+N_1(\ell_1)\bigr)^2,
  \]
  donc $\Srp_\sigma(\ell_1)=1=\Srp_\sigma(\ell_1+1)$. De même, si
  $\ell_2\in L_2$ est tel que $T_2(\ell_2)=1$, alors
  \[
    \Srp_\sigma(\ell_2)=\Srp_\sigma(\ell_2+1)=1;
  \]
  de plus, $\ell_1+\ell_2\in L_3$ et $T_3(\ell_1+\ell_2)=1$, donc
  \[
    \Srp_\sigma(\ell_1+\ell_2)=\Srp_\sigma(\ell_1+\ell_2+1)=1.
  \]
  Ces égalités entraînent que
  $\Srp_\sigma(\ell_1+\ell_2)-\Srp_\sigma(\ell_1) -
  \Srp_\sigma(\ell_2)=1$, donc la restriction de $\Srp_\sigma$ au
  sous-espace de $L$ engendré par $\ell_1$ et $\ell_2$ est une forme
  quadratique isométrique à $[1,1]$. L'orthogonal de ce sous-espace
  contient~$1$ car $\Srp_\sigma(1)=0$, $\Srp_\sigma(\ell_1+1) =
  \Srp_\sigma(\ell_1)$ et
  $\Srp_\sigma(\ell_2+1)=\Srp_\sigma(\ell_2)$. Comme~$1$ est isotrope,
  on en déduit
  \[
    \Srp_\sigma\rvert_L\simeq [1,1]\perp [0,0].
    \qedhere
  \]
\end{proof}

\begin{exam}
  \label{exam:ind2}
  Supposons $A=\End_QV$ pour $V$ un espace vectoriel de dimension~$4$
  sur un corps de quaternions $Q$. Alors $\sigma$ est adjointe à une
  forme hermitienne alternée $h$ sur $V$ pour l'involution
  canonique~$\invo$ sur $Q$: voir \cite[(4.2)]{BoI}. On peut supposer
  que $h$ représente~$1$ car $h$ n'est
  déterminée qu'à un facteur scalaire près, et choisir une base de $V$
  par rapport à laquelle 
  $h$ est diagonale, soit $h=\langle1,u_1,u_2,u_3\rangle$ pour
  certains 
  $u_1$, $u_2$, $u_3\in F^\times$. On peut prendre pour $L$ la
  $F$-algèbre déployée engendrée par les projections $p_1$, $p_2$,
  $p_3$, $p_4$ sur les vecteurs de base. Si $\alpha_1$ échange $p_1$
  et $p_2$, ainsi que $p_3$ et $p_4$, alors un calcul semblable à
  celui de la démonstration de la proposition~\ref{prop:decorth}
  montre que dans 
  la représentation matricielle par rapport à la base choisie,
  \[
    W_1=\left\{
      \begin{pmatrix}
        0&u_1x&&\\ \overline x&0&&\\ &&0&u_3y\\ &&u_2
        \overline{y}&0 
      \end{pmatrix}
      \mid x, y\in Q
    \right\}.
  \]
  Dès lors, en désignant par $n_Q$ la forme norme réduite de~$Q$, on
  obtient $\Srp_\sigma\rvert_{W_1}\simeq\langle u_1,u_2u_3\rangle
  n_Q$. De même, si $\alpha_2$ échange $p_1$ et $p_3$ ainsi que $p_2$
  et $p_4$, alors $\Srp_\sigma\rvert_{W_2}\simeq \langle
  u_2,u_1u_3\rangle n_Q$ et $\Srp_\sigma\rvert_{W_3}\simeq \langle
  u_3, u_1u_2\rangle n_Q$. Par conséquent,
  \[
    \pi_3=\langle1,u_1u_2u_3\rangle n_Q \qquad\text{et}\qquad
    \pi_5=\langle1,u_1,u_2,u_3\rangle \pi_3.
  \]
  Ce résultat corrobore le calcul dans~\cite[Prop.~8.6]{BGBT3} de la
  forme de Pfister discriminante $\pi_3$.
\end{exam}

Revenant au cas général, on caractérise les cas où $\pi_3$ est
hyperbolique.

\begin{prop}
  \label{prop:pi3}
  Les conditions suivantes sont équivalentes:
  \begin{enumerate}
  \item[\emph{(i)}]
    $\pi_3$ est hyperbolique;
  \item[\emph{(ii)}]
    $(A,\sigma)\simeq(H_1,\sigma_1)\otimes (H_2,\sigma_2) \otimes
    (H_3,\sigma_3)$ pour certaines algèbres de quaternions à
    involution $(H_i,\sigma_i)$;
  \item[\emph{(iii)}]
    $(A,\sigma)\simeq(Q_1,\invo) \otimes (Q_2,\invo) \otimes
    (Q_3,\invo)$ pour certaines algèbres de quaternions $Q_i$ avec
    leur involution canonique.
  \end{enumerate}
\end{prop}

\begin{proof}
  (i)~$\Rightarrow$~(ii) On choisit dans $\Symd(\sigma)$ une
  $F$-algèbre biquadratique $L$ comme dans la
  section~\ref{sec:orth}. Sous 
  l'hypothèse~(i), la forme $\Srp_\sigma\rvert_{W_1}$ est isotrope,
  donc on peut trouver $x\in W_1$ non nul tel que $T_1(x^2)=0$, ce qui
  revient à dire que $x^2\in F$. Si $x^2=0$, alors la forme
  quadratique $\sq_1\colon W_1\to L_1$ d'élévation au carré est
  isotrope; mais la proposition~\ref{prop:decorth} montre que cette
  forme est non singulière, donc on peut trouver $y\in W_1$ tel que
  $xy+yx\in L_1^\times$. Alors pour $\lambda=(y^2+1)(xy+yx)^{-1}\in
  L_1$ on a $(x\lambda+y)^2=1$. Quitte à changer $x$, on peut donc
  toujours supposer $x^2\in 
  F^\times$. Comme la conjugaison par $x$ induit sur $L_2$
  l'automorphisme non trivial, $x$ et $L_2$ engendrent une algèbre de
  quaternions $H_1$ stable sous $\sigma$. La restriction $\sigma_1$ de
  $\sigma$ à $H_1$ est orthogonale: en effet, si $\ell\in L_2$
  satisfait $T_2(\ell)=1$, alors $x\ell+\ell x=x\ell +
  \sigma_1(x\ell)=x$, ce qui montre que $x$ est dans
  $\Symd(\sigma_1)$, donc $\Symd(\sigma_1)\neq F$. Le centralisateur
  de $H_1$ dans $A$ est une  
  algèbre simple centrale de degré~$4$ sur laquelle $\sigma$ se
  restreint en une involution symplectique,
  d'après~\cite[(2.23)]{BoI}. Pour établir~(ii), il suffit maintenant
  d'observer que ce centralisateur se décompose en produit de deux
  algèbres de quaternions stables sous l'involution; cela résulte
  de~\cite[(16.16)]{BoI} si la restriction de $\sigma$ n'est pas
  hyperbolique, et de~\cite[Prop.~5.3]{BGBT3} si elle l'est.

  (ii)~$\Rightarrow$~(iii) Sous l'hypothèse~(ii), l'une au moins des
  involutions $\sigma_1$, $\sigma_2$, $\sigma_3$ doit être
  symplectique, vu \cite[(2.23)]{BoI}. Disons $\sigma_1=\invo$ et
  supposons que $\sigma_2$ soit
  orthogonale. D'après~\cite[(2.7)]{BoI}, il existe un élément
  inversible non central $j_2\in H_2$ tel que $\overline{j_2}=j_2$ et
  $\sigma_2(h)=j_2\overline h j_2^{-1}$ pour tout $h\in H_2$. Ces
  conditions entraînent $j_2^2\in F^\times$, $\sigma_2(j_2)=j_2$, et
  pour $i_2\in H_2$ tel que $j_2i_2j_2^{-1}=i_2+1$, on a
  $\sigma_2(i_2)=i_2$. Choisissons encore $i_1$, $j_1\in H_1$ tels que
  $j_1^2\in F^\times$ et $j_1i_1j_1^{-1}=i_1+1$, d'où
  $\sigma_1(i_1)=i_1+1$ puisque $\sigma_1$ est la conjugaison
  quaternionienne. 
  Alors $i_1$ et $j_1j_2$ (resp.\
  $i_1+i_2$ et 
  $j_2$) engendrent une algèbre de quaternions $Q_1$ (resp.\ $Q_2$) et
  $(H_1,\sigma_1)\otimes(H_2,\sigma_2) = (Q_1,\invo)\otimes
  (Q_2,\invo)$. Si $\sigma_3$ est orthogonale, on répète l'argument
  pour trouver une nouvelle décomposition de $(A,\sigma)$ en produit
  d'algèbres de quaternions sur chacune desquelles la restriction de
  $\sigma$ est l'involution canonique.

  (iii)~$\Rightarrow$~(i) Pour $k=1$, $2$, $3$, soient $i_k$, $j_k\in
  Q_k$ tels que $j_k^2\in F^\times$ et $j_ki_kj_k^{-1}=i_k+1$. Alors
  $L=F(i_1+i_2, i_2+i_3, i_3+i_1)$ est une algèbre biquadratique comme
  dans la section~\ref{sec:orth}. Si $\alpha_1$ est l'automorphisme
  non trivial de $L$ qui fixe $i_2+i_3$, alors $j_1\in W_1$ et
  $T_1(j_1^2)=0$, donc $\Srp_\sigma\rvert_{W_1}$ est isotrope, et il
  en est de même de $\pi_3$. Comme $\pi_3$ est une forme de Pfister,
  elle est hyperbolique.
\end{proof}

Comme toute involution symplectique hyperbolique est décomposable
(voir~\cite[Prop.~6.1]{BGBT3}), on tire immédiatement de la
proposition précédente:

\begin{coro}
  \label{coro:pi3hyp}
  Si $(A,\sigma)$ est hyperbolique, alors $\pi_3$ est hyperbolique.
\end{coro}

On caractérise enfin les cas où $\pi_5$ est hyperbolique, en
commençant par l'observation suivante:

\begin{lemm}
  \label{lemm:pi5isot}
  Si $(A,\sigma)$ est isotrope, alors $\pi_5$ est hyperbolique.
\end{lemm}

\begin{proof}
  Si $(A,\sigma)$ est isotrope, alors $A$ n'est pas à division. Si son
  indice est~$1$ ou son co-indice est~$2$, alors
  $(A,\sigma)$ est hyperbolique, et le corollaire~\ref{coro:pi3hyp}
  montre que $\pi_3$ est hyperbolique. Il en est donc de même de
  $\pi_5$, puisque $\pi_5$ est multiple de $\pi_3$. Il ne reste donc
  qu'à étudier le cas où l'indice de $A$ est~$2$. On est alors dans la
  situation de l'exemple~\ref{exam:ind2}, où l'on a vu que $\pi_5$ est
  un multiple de $\langle1,u_1,u_2,u_3\rangle n_Q$. Or, cette dernière
  forme quadratique est la forme diagonale $h(X,X)$ pour la forme
  hermitienne 
  $h=\langle1,u_1,u_2,u_3\rangle$ à laquelle $\sigma$ est
  adjointe. Puisque $(A,\sigma)$ est isotrope, 
  la forme $h$ est isotrope, donc aussi sa forme diagonale, ce qui
  entraîne que $\pi_5$ est isotrope, donc hyperbolique puisque c'est
  une forme de Pfister.
\end{proof}

Remarquons que les réciproques du corollaire~\ref{coro:pi3hyp} et du
lemme~\ref{lemm:pi5isot} ne valent pas: si les 
conditions de la proposition~\ref{prop:pi3} sont satisfaites, alors
$\pi_3$ est hyperbolique, donc $\pi_5$ l'est aussi, puisque $\pi_5$
est multiple de $\pi_3$; mais l'algèbre $A$ peut être à division.

\begin{prop}
  \label{prop:pi5}
  La forme $\pi_5$ est hyperbolique si et seulement si $\Symd(\sigma)$
  contient un élément $x\notin F$ tel que $x^2\in F$.
\end{prop}

\begin{proof}
  Supposons d'abord $\pi_5$ hyperbolique, et choisissons une
  $F$-algèbre biquadratique $L\subset\Symd(\sigma)$ comme dans la
  section~\ref{sec:orth} pour représenter
  $\pi_5=\langle1,a_1,a_2,a_1a_2\rangle\pi_3$ suivant le
  corollaire~\ref{coro:comp}. Comme $\pi_5$ contient la forme
  $\langle1\rangle\perp 
  \Srp_\sigma\rvert_{W_1}\perp \Srp_\sigma\rvert_{W_2}$, dont la
  dimension dépasse la moitié de celle de $\pi_5$, il existe $y\in
  W_1\oplus W_2$ tel que $\Srp_\sigma(y)=1$. Vu la
  proposition~\ref{prop:decorth}, on a $\Trp_\sigma(y)=0$, donc
  $\Prp_{\sigma,y}(X) = X^4+X^2+U(y)X+\Nrp_\sigma(y)$ pour une
  certaine forme cubique $U\colon \Symd(\sigma)\to F$. Si $\ell\in
  L_3$ est tel que $T_3(\ell)=1$, alors $y\ell+\ell y=y$ et
  $y^2\ell=\ell y^2$. En utilisant ces relations, on voit que
  $U(y)y=\Prp_{\sigma,y}(y)\ell + \ell\Prp_{\sigma,y}(y)$, donc
  $U(y)y=0$ et par conséquent $y^4+y^2=\Nrp_\sigma(y)\in F$.

  Si $y^2\in F$, disons $y^2=\lambda\in F$, alors le polynôme minimal
  de $y$ sur $F$ est $X^2-\lambda$, donc $\Prp_{\sigma,y}(X) =
  (X^2-\lambda)^2$, ce qui est incompatible avec
  $\Srp_\sigma(y)=1$. Par conséquent, $y^2\notin F$. De plus,
  $y^2\notin L_3$ puisque $y$ commute avec 
  $y^2$ mais ne centralise pas $L_3$. Dès lors, $L_3[y^2] =
  L_3\otimes_FF[y^2]$ est une $F$-algèbre biquadratique étale contenue
  dans $\Symd(\sigma)$. Vu sa dimension, c'est une $F$-algèbre étale
  maximale dans $\Symd(\sigma)$, donc par~\cite[Prop.~5.6]{BGBT1} elle
  satisfait 
  les conditions énoncées au début de la section~\ref{sec:orth}. On la
  note $L'$ et on s'autorise à utiliser en référence à $L'$ les mêmes
  structures que celles définies pour $L$, en affectant leur notation
  d'un $'$ pour les distinguer. Si la numérotation des éléments du
  groupe de Galois de $L'$ est telle que $F[y^2]=L'_1$, alors $y\in
  W'_1$. Pour tout $w'_2\in W'_2$ on a $y*w'_2\in W'_3$ et
  $\Srp_\sigma(y*w'_2)=\Srp_\sigma(w'_2)$ par la
  proposition~\ref{prop:compo}. Choisissant $w'_2\neq0$, on pose
  $z=w'_2+(y*w'_2)$. Alors $\Srp_\sigma(z) =
  \Srp_\sigma(w'_2)+\Srp_\sigma(y*w'_2)=0$ puisque $W'_2$ et $W'_3$
  sont orthogonaux. Les mêmes arguments que pour $y$ donnent
  $\Trp_\sigma(z)=U(z)=0$, donc $z^4=\Nrp_\sigma(z)\in F$. Si $z^2\in
  F$, on choisit $x=z$; si $z^2\notin F$, on prend $x=z^2$. Dans un
  cas comme dans l'autre, on a $x\notin F$ et $x^2\in F$.

  Réciproquement, supposons donné $x\in\Symd(\sigma)\setminus F$ tel
  que $x^2\in F$. Si $x^2=\lambda^2$ pour un $\lambda\in F$, alors
  $\sigma(x-\lambda)\cdot(x-\lambda) = (x-\lambda)^2=0$, donc
  $(A,\sigma)$ est isotrope. Dans ce cas, le lemme~\ref{lemm:pi5isot}
  montre que $\pi_5$ est hyperbolique. Pour la suite du raisonnement,
  on peut donc supposer que $F(x)$ est un corps, extension quadratique
  purement 
  inséparable de $F$. On note $K=F(x)$ et $B$ le
  centralisateur de $K$ dans $A$, qui est une algèbre simple centrale
  de degré~$4$ sur $K$. Comme $K\subset\Symd(\sigma)$, la
  restriction $\sigma_B$ de $\sigma$ à $B$ est une involution
  symplectique, vu~\cite[Prop.~4.5]{RST}. Par \cite[(16.16)]{BoI} si
  $\sigma_B$ n'est pas hyperbolique et par \cite[Prop.~5.3]{BGBT3} si
  $\sigma_B$ est hyperbolique, on
  peut donc décomposer:
  \[
    (B,\sigma_B) = (Q_1,\sigma_1)\otimes_K(Q_2,\invo),
  \]
  où $Q_1$ et $Q_2$ sont des algèbres de quaternions sur $K$ et
  $\sigma_1$ est une involution orthogonale sur $Q_1$. On peut alors
  trouver dans $Q_1$ des éléments $\sigma_1$-symétriques $i_1$, $j_1$
  tels que $i_1^2+i_1\in K$, $j_1^2\in K^\times$ et
  $j_1i_1j_1^{-1}=i_1+1$.  Quitte à remplacer $i_1$ par $i_1^2$, on
  peut supposer $i_1^2+i_1\in F$. Alors $F[i_1]$ est une $F$-algèbre
  quadratique étale; elle est contenue dans $\Symd(\sigma)$ car si
  $i_2\in Q_2$ satisfait $\overline{i_2}=i_2+1$, alors
  $i_1=i_1i_2+\sigma(i_1i_2)$. De plus, $A$ est un module à droite (ou
  à gauche) libre sur $F[i_1]$, car l'automorphisme intérieur induit
  par $j_1$ se restreint en l'automorphisme non trivial de $F[i_1]$
  sur $F$. Cela montre que $F[i_1]$ est une sous-algèbre
  \guillemotleft\,nette\,\guillemotright\ de $(A,\sigma)$, dans la
  terminologie de~\cite[\S5]{BGBT3}. Dès lors,
  par~\cite[Th.~6.10]{BGBT1}, on peut plonger 
  $F[i_1]$ dans une $F$-algèbre biquadratique $L$ comme dans la
  section~\ref{sec:orth}, de sorte que $F[i_1]$ soit la sous-algèbre
  fixe sous un élément (disons $\alpha_1$) du groupe de Galois. On a
  aussi $j_1\in\Symd(\sigma)$ car $j_1=j_1i_2+\sigma(j_1i_2)$, et
  comme $j_1i_1=(i_1+1)j_1$ on doit avoir $j_1\in W_2\oplus W_3$, avec
  les notations de la section~\ref{sec:orth}. Comme pour les éléments
  $y$ et $z$ de la première partie de la preuve, on tire
  $j_1^4+\Srp_\sigma(j_1)j_1^2+\Nrp_\sigma(j_1)=0$. Mais $j_1^4\in
  F^\times$ puisque $j_1^2\in K^\times$, donc
  $\Srp_\sigma(j_1)j_1^2\in F$. Si $j_1^2\notin F$, on en déduit
  $\Srp_\sigma(j_1)=0$; mais si $j_1^2\in F$, soit $j_1^2=\lambda\in
  F$, alors $\Prp_{\sigma, j_1}(X)=(X^2-\lambda)^2$ donc on a aussi
  $\Srp_\sigma(j_1)=0$. La restriction de $\Srp_\sigma$ à $W_2\oplus
  W_3$ est donc isotrope; comme cette restriction est une sous-forme
  de $\pi_5$, cela
  entraîne que $\pi_5$ est isotrope, donc hyperbolique puisque c'est
  une forme de Pfister.
\end{proof}

\section{Involutions unitaires et orthogonales}

Dans cette section, $F$ est un corps de caractéristique~$2$. On
désigne par $B$ une algèbre d'Azumaya de degré~$4$ sur une $F$-algèbre
quadratique étale~$Z$ et par $\tau$ une involution unitaire sur $B$
qui fixe~$F$. Soit $\Sym(\tau)\subset B$ le $F$-espace vectoriel des
éléments symétriques sous $\tau$ et $\Srd_\tau$ la restriction de la
\guillemotleft\,seconde trace\guillemotright\ $\Srd_B$ à $\Sym(\tau)$;
c'est une forme quadratique $\Srd_\tau\colon\Sym(\tau)\to F$.

\begin{theo}
  \label{th:unit}
  Il existe une $2$-forme quadratique de Pfister $\pi_2$ et une
  $4$-forme quadratique de Pfister $\pi_4$ déterminées de manière
  unique à isométrie près par la propriété suivante: la classe de
  $\Srd_\tau$ dans le groupe de Witt se décompose comme suit:
  \[
    \Srd_\tau=[1,1]+\pi_2+\pi_4.
  \]
  De plus,
  \begin{enumerate}
  \item[\emph{(i)}]
    $\pi_4$ est multiple de $\pi_2$, c'est-à-dire qu'il existe $a_1$,
  $a_2\in F^\times$ tels que
  $\pi_4=\langle1,a_1,a_2,a_1a_2\rangle\pi_2$;
\item[\emph{(ii)}]
  la forme $\pi_2$ est hyperbolique si et seulement si l'algèbre
  $B$ se décompose en produit tensoriel d'algèbres de
  quaternions stables sous $\tau$;
\item[\emph{(iii)}]
  la forme $\pi_4$ est hyperbolique si et seulement si $\Sym(\tau)$
  contient un élément non central dont le carré est central.
  \end{enumerate}
\end{theo}

La démonstration est en tout point semblable à celle du
théorème~\ref{th:symp}: il suffit d'y remplacer $\Symd(\sigma)$ par
$\Sym(\tau)$, $\Srp_\sigma$ par $\Srd_\tau$ et l'algèbre de
quaternions déployée $Q$ par une $F$-algèbre quadratique étale
déployée. L'existence d'une $F$-algèbre biquadratique étale
$L\subset\Sym(\tau)$ est assurée
par~\cite[Th.~7.4]{BGBT1} et la relation entre compositions d'espaces
quadratiques de dimension~$4$ et $2$-formes quadratiques de Pfister
est explicitée dans~\cite[Prop.~3.9]{BT}. D'ailleurs, le cas où le
centre $Z$ est déployé 
(et donc $B\simeq E\times E^{\text{op}}$ pour une $F$-algèbre simple
centrale $E$ de degré~$4$) est déjà traité dans~\cite{T2trace}.
\medbreak

Le cas des involutions orthogonales requiert des ajustements plus
substantiels car les formes quadratiques en jeu sont
singulières. Suivant~\cite[D.3.2]{K}, on dit qu'une forme quadratique
est \emph{totalement singulière} si sa forme polaire est identiquement
nulle. Les formes totalement singulières sont donc les formes
diagonales $b(X,X)$ des formes bilinéaires symétriques $b$. Une forme
totalement singulière est appelée \emph{quasi $k$-forme de Pfister} si
c'est la forme diagonale d'une $k$-forme de Pfister bilinéaire,
voir~\cite[D.11.1]{K}.

Soit $C$ une $F$-algèbre simple centrale de degré~$4$ et $\rho$ une
involution orthogonale sur $C$. Soit encore $\Sym(\rho)$ l'espace des
éléments symétriques sous $\rho$ et $\Srd_\rho\colon \Sym(\rho)\to F$
la restriction de $\Srd_C$ à $\Sym(\rho)$. On sait définir le
déterminant $\det\rho\in F^\times/F^{\times2}$:
voir~\cite[(7.2)]{BoI}. Soit 
$\pi'_1$ la quasi $1$-forme de Pfister diagonale de
$\langle1,\delta\rangle$, où $\delta\in F^\times$ représente
$\det\rho$. 

\begin{theo}
  \label{th:orth}
  La forme quadratique $\Srd_\rho$ se décompose en somme orthogonale
  \[
    \Srd_\rho\simeq[0,0]\perp[1,1]\perp\varphi
  \]
  où $\varphi$ est une forme quadratique totalement singulière de
  dimension~$6$. La forme $\pi'_3=\pi'_1\perp\varphi$ est une
  quasi $3$-forme de Pfister et les formes $\varphi$ et $\pi'_3$ sont
  déterminées de manière unique à isométrie 
  près par $(C,\rho)$. De plus,
  \begin{enumerate}
  \item[\emph{(i)}]
    $\pi'_3$ est multiple de $\pi'_1$, c'est-à-dire qu'il existe
    $a_1$, $a_2\in F^\times$ tels que
  $\pi'_3=\langle1,a_1,a_2,a_1a_2\rangle\pi'_1$;
\item[\emph{(ii)}]
  la forme $\pi'_1$ est quasi-hyperbolique si et seulement si
  l'algèbre 
  $C$ se décompose en produit tensoriel d'algèbres de
  quaternions stables sous $\rho$;
\item[\emph{(iii)}]
  la forme $\pi'_3$ est hyperbolique si et seulement si $\Sym(\rho)$
  contient un élément non central dont le carré est central.
  \end{enumerate}
\end{theo}

\begin{proof}
  Encore une fois, \cite[Th.~7.4]{BGBT3} donne une $F$-algèbre
  biquadratique étale $L\subset\Sym(\rho)$, et les mêmes arguments que
  dans la proposition~\ref{prop:decorth} établissent que
  $\Sym(\rho)=L\stackrel{\perp}{\oplus} W_1 \stackrel{\perp}{\oplus}
  W_2 \stackrel{\perp}{\oplus} W_3$ et que $\Srd_\rho(x)=T_i(x^2)$
  pour $x\in W_i$. À présent, $W_i$ est un module libre de rang~$1$
  sur $L_i$. Lorsque $(C,\rho)$ est isomorphe à $(M_4(F),t)$, où $t$
  est la transposition---ce qui est toujours le cas si $F$ est
  fini---on peut choisir pour $L$ l'algèbre diagonale et retrouver la
  même situation que dans la démonstration de la
  proposition~\ref{prop:decorth}, avec $F$ à la place de $Q$. On voit
  que si $w_i\in W_i$ est de déterminant non nul, alors $W_i=w_iL_i$,
  pour $i=1$, $2$, $3$. Dans le cas général, pour $F$ infini, un
  argument de densité établit l'existence de $w_i\in W_i$ tel que
  $\Nrd_C(w_i)\neq0$, ce qui entraîne encore $W_i=w_iL_i$ pour $i=1$,
  $2$, $3$. Pour le calcul de $\Srd_\rho\rvert_{W_i}$, on peut de plus
  supposer $\Srd_\rho(w_i)\neq0$. Un analogue de la
  proposition~\ref{prop:decorth} donne 
  $\Srd_\rho(w_ix) = T_i(w_i^2x^2)$ pour $x\in L_i$, donc
  $\Srd_\rho\rvert_{W_i}$ est 
  isométrique au transfert suivant $T_i$ de la forme $w_i^2X^2$ sur
  $L_i$. La forme $\Srd_\rho\rvert_{W_i}$ est donc totalement
  singulière; par rapport à la base $w_i$, $w_i(\Srd_\rho(w_i)+w_i^2)$
  de $W_i$ on calcule qu'elle se diagonalise en
  $\langle\Srd_\rho(w_i),\, \Srd_\rho(w_i)N_i(w_i^2)\rangle
  =\langle\Srd_\rho(w_i)\rangle \langle1,\Nrd_{C}(w_i)\rangle$. En 
  examinant le cas où $L$ est déployée on voit que
  $W_i\subset\Symd(\rho)$, donc $\Nrd_C(w_i)$ représente $\det\rho$
  d'après la définition donnée en~\cite[(7.2)]{BoI}. Dès lors,
  \[
    \Srd_\rho\rvert_{W_i}\simeq \langle\Srd_\rho(w_i)\rangle \pi'_1
    \qquad\text{pour $i=1$, $2$, $3$}.
  \]
  Soit $a_i=\Srd_\rho(w_i)$ pour $i=1$, $2$. Comme dans la
  proposition~\ref{prop:compo}, on observe que $w_1w_2+w_2w_1\in W_3$
  et $\Srd_\rho(w_1w_2+w_2w_1)=\Srd_\rho(w_1)\Srd_\rho(w_2)$. Dès
  lors, $\Srd_\rho\rvert_{W_3}$ représente~$a_1a_2$, et
  $\Srd_\rho\rvert_{W_3} \simeq\langle a_1a_2\rangle\pi'_1$. Par
  ailleurs, $\Srd_\rho\rvert_L\simeq[0,0]\perp[1,1]$ comme
  précédemment, donc la décomposition orthogonale de $\Sym(\rho)$
  donne
  \[
    \Srd_\rho\simeq[0,0]\perp[1,1]\perp \langle a_1\rangle \pi'_1
    \perp \langle a_2\rangle \pi'_1 \perp \langle a_1a_2\rangle\pi'_1.
  \]
  On obtient ainsi la décomposition souhaitée, avec $\varphi=\langle
  a_1,a_2,a_1a_2\rangle\pi'_1$ et donc
  \[
    \pi'_3=\langle1,a_1,a_2,a_1a_2\rangle\pi'_1.
  \]
  La forme $\varphi$
  est déterminée de manière unique, puisque c'est la restriction de
  $\Srd_\rho$ au radical de sa forme polaire. Les énoncés~(ii) et
  (iii) se démontrent en adaptant le raisonnement des
  propositions~\ref{prop:pi3} et \ref{prop:pi5}.
\end{proof}

L'énoncé~(ii) a été démontré par d'autres arguments
dans~\cite[Prop.~3.7]{KPS}. 

\bibliographystyle{amsplain}

\bibliography{Srp}

\end{document}